\documentclass[11pt,oneside,reqno,fleqn]{amsart}
\usepackage{}
\usepackage{amssymb}

\usepackage{bbm}
\usepackage{cases}
\usepackage{amsmath}
\usepackage{graphicx}
\usepackage{mathrsfs}
\usepackage{stmaryrd}
\usepackage{color}
\usepackage{soul}
\usepackage[dvipsnames]{xcolor}
\usepackage{amsfonts}
\usepackage{amsmath,amssymb,amsthm}
\usepackage[colorlinks,citecolor=blue,urlcolor=blue]{hyperref}
\usepackage[shortlabels] {enumitem}
\usepackage[nameinlink,capitalize]{cleveref}
\usepackage{microtype}
\usepackage{todonotes}
\numberwithin{equation}{section}
\newcommand{\be}{\begin{eqnarray}}
\newcommand{\ee}{\end{eqnarray}}
\newcommand{\ce}{\begin{eqnarray*}}
\newcommand{\de}{\end{eqnarray*}}
\newtheorem{theorem}{Theorem}[section]
\newtheorem{lemma}[theorem]{Lemma}
\newtheorem{proposition}[theorem]{Proposition}
\newtheorem{corollary}[theorem]{Corollary}
\theoremstyle{remark}
\newtheorem{assumption}[theorem]{Assumption}
\newtheorem{example}[theorem]{Example}
\newtheorem{remark}[theorem]{Remark}
\newtheorem{definition}[theorem]{Definition}

\def\p{\partial}

\def\[{{\Big[}}
\def\]{{\Big]}}
\def\<{{\langle}}
\def\>{{\rangle}}
\def\({{\Big(}}
\def\){{\Big)}}

\def\bx{{\mathbf{x}}}

\def\dif{{\mathord{{\rm d}}}}

\def\={&\!\!=\!\!&}
\def\bt{\begin{theorem}}
\def\et{\end{theorem}}
\def\bl{\begin{lemma}}
\def\el{\end{lemma}}
\def\br{\begin{remark}}
\def\er{\end{remark}}

\def\bd{\begin{definition}}
\def\ed{\end{definition}}
\def\bp{\begin{proposition}}
\def\ep{\end{proposition}}
\def\bc{\begin{corollary}}
\def\ec{\end{corollary}}
\def\bx{\begin{example}}
\def\ex{\end{example}}

\def\cM{{\mathcal M}}

\def\mE{{\mathbb E}}

\def\mH{{\mathbb H}}
\def\mI{{\mathbb I}}

\def\mL{{\mathbb L}}

\def\mP{{\mathbb P}}

\def\mR{{\mathbb R}}

\def\sF{{\mathscr F}}

\def\sH{{\mathscr H}}

\def\geq{\geqslant}
\def\leq{\leqslant}

\allowdisplaybreaks

\begin{document}
	\title{A Wong-Zakai theorem for SDEs with singular drift}
	
	\date{}
	
\author{Chengcheng Ling, Sebastian Riedel and Michael Scheutzow}

\address{Chengcheng Ling:
Technische Universit\"at Berlin,
Fakult\"at II, Institut f\"ur Mathematik,
10623 Berlin, Germany
\\
Email: ling@math.tu-berlin.de
 }

\address{Sebastian Riedel:
Leibniz Universit\"at Hannover,
Institute of Analysis,
30167 Hannover, Germany
\\
Email: riedel@math.uni-hannover.de
}

\address{Michael Scheutzow:
Technische Universit\"at Berlin,
Fakult\"at II, Institut f\"ur Mathematik,
10623 Berlin, Germany
\\
Email: ms@math.tu-berlin.de
}

	\begin{abstract}
	We study stochastic differential equations (SDEs) with multiplicative Stratonovich-type noise of the form
\begin{align}\label{SDE}
dX_t = b(X_t) \, \dif t + \sigma(X_t)\circ\dif W_t,\quad X_0=x_0\in\mathbb{R}^d, \quad t\geq0,
\end{align}
with a possibly singular drift  $b\in L^{{p}}(\mathbb{R}^d)$, $p>d$ and $p\geq 2$, and show that such SDEs can be approximated by random ordinary differential equations by smoothing the noise and the singular drift at the same time. We further prove a support theorem for this class of SDEs in a rather simple way using the Girsanov theorem.

		
		\bigskip
		
		\noindent {{\bf AMS 2020 Mathematics Subject Classification:} 60H10, 60F15, 60J60.}
		
		\noindent{{\bf Keywords:} singular stochastic differential equations; stability with respect to singular drifts; Wong-Zakai approximation; Zvonkin's transformation; support theorem.}
	\end{abstract}
	
	\maketitle
	
\section{Introduction}
Consider the following stochastic differential equation (abbreviated as SDE) in $\mR^d$:
\begin{align}\label{sde00}
\dif X_t = b(t,X_t) \,\dif t + \sigma(t,X_t)\,\dif W_t,\quad X_0=x_0\in\mathbb{R}^d,\quad t\geq0,
\end{align}
where $d\geq1$, $b=(b^{(i)})_{1\leq i\leq d}: \mR_+\times\mR^d\rightarrow\mR^d $ and $\sigma=(\sigma_{ij})_{1\leq i,j\leq d}: \mR_+\times\mR^d\rightarrow L(\mathbb{R}^d)$ $(:=d\times d$ real valued matrices$)$ are Borel measurable, and $(W_t)_{t\geq0}$ is a standard $d$-dimensional Brownian motion defined on some filtered probability space $(\Omega,\sF,(\sF_t)_{t\geq 0},\mP)$. Using Picard iteration, it is easy to see that the corresponding ordinary differential equation (abbreviated as ODE)
 \begin{align}\label{ode00}
\dif X_t=b(t,X_t) \, \dif t,\quad X_0=x_0\in\mathbb{R}^d, \quad t\geq0,
\end{align}
has a unique solution provided that $b$ is Lipschitz continuous. With a similar argument, under suitable growth and regularity assumptions on $\sigma$, it can be shown that there exists a unique strong solution to the SDE \eqref{sde00}, assuming again that $b$ is Lipschitz continuous. On the other hand, when $b$ fails to be Lipschitz continuous, e.g. only $\alpha$-H\"older with $\alpha \in (0,1)$, it is well known that the ODE \eqref{ode00} is generally ill-posed. Interestingly, due to a 'regularization by noise' effect, the SDE \eqref{sde00} defined with a very singular $b$, maybe even discontinuous, might still be well-posed and possesses a unique solution even in the strong sense. There are many works on this aspect. In the following, we mention a few of them. A remarkable result due to Zvonkin \cite{Z} shows that if $d=1$, $\sigma$ is uniformly elliptic, $b$ and $\sigma$ are bounded, and $\sigma$ is continuous with respect to space and time variables,  then the SDE (\ref{sde00}) admits a unique strong solution $(X_t(x_0))_{t\geq0}$ for each $x_0\in\mR$. The key idea is to use a certain transformation of the SDE, nowadays called \emph{Zvonkin's transformation} in the literature.  Zvonkin's result was then extended to the multidimensional case by Veretennikov \cite{Ve}. 
A further generalization was obtained by  Krylov and  R\"ockner \cite{Kr-Ro} where existence and uniqueness of a strong solution for the SDE (\ref{sde00}) was shown when $\sigma$ is identity matrix in $\mR^d$ (i.e. the SDE is driven by additive noise) and 
\begin{align}
b\in L^q_{loc}(\mR_+;L^p_{loc}(\mR^d))\quad\text{with}\quad q,p\in[2,\infty)\quad\text{and}\quad  d/p+2/q<1.    \label{pq}
\end{align}
 Besides, Fedrizzi and Flandoli \cite{FF1} introduced a new method to prove existence and uniqueness of a global strong solution to the SDE \eqref{sde00} by using regularizing properties of solutions to the Kolmogorov equation corresponding to \eqref{sde00}, assuming that \eqref{pq} holds globally. For a well-posedness result for \eqref{sde00} driven by {multiplicative noise} (i.e. $\sigma$ is not a constant matrix only), we refer to Zhang \cite{Zhang2011} who applied {Zvonkin's transformation} and Von der L\"uhe \cite{K} where Fedrizzi and Flandoli's method in \cite{FF1} was generalized to the multiplicative noise case. More recently, Zhang and Zhao \cite{ZZ} obtained a well-posedness result for the martingale problem induced by the SDE \eqref{sde00} when the drift is a Schwartz distribution using Zvonkin's transformation. Another way to investigate singular SDEs is rough paths theory: Harang and Perkowski in \cite{FHP},  Catellier and   Gubinelli in \cite{CG}, and Galeati and Gubinelli in \cite{GG} showed the pathwise uniqueness for a large class of singular SDEs.\\

 A Wong-Zakai approximation to an SDE is given by a series of random ODEs which converge to the original equation. For classical SDEs, the approximating random ODEs can be obtained by just smoothing the noise. The first results for classical SDEs are contained in the works by Wong and Zakai \cite{WZK1,WZK2}. In \cite[Theorem 2]{WZK1}, it was proved that in dimension one, under the assumptions that $b \in\mathcal{C}_b^1([0,T]\times\mathbb{R})$ and $\sigma \in\mathcal{C}_b^2([0,T]\times\mathbb{R})$, if $(w^n_t)_{n\geq1}$ is a continuous and piecewise smooth sequence of approximations to the Brownian motion $(W_t)_{t\geq0}$, then the solutions $X_t^{w^n}$ to the following equation
 $$\dif X_t^{w^n} = b(t,X_t^{w^n}) \, \dif t + \sigma(t,X_t^{w^n}) \, \dif w^n_t,\quad X_0=x_0\in\mathbb{R},\quad t\geq0,$$
converge to $X_t$ almost surely,  where $X_t$ solves
\begin{align}\label{sde00stro}
\dif X_t=b(t,X_t)\dif t + \sigma(t,X_t)\circ\dif W_t,\quad X_0=x_0\in\mathbb{R},\quad t\geq0.
\end{align}
Here $\sigma(t,X_t)\circ \dif W_t:=\sigma(t,X_t) \, \dif W_t + \frac{1}{2}\sigma(t,X_t)\frac{\partial \sigma(t,X_t)}{\partial x} \, \dif t$, i.e. the $\circ$ integral denotes the usual Stratonovich integral, and $\frac{1}{2}\sigma(t,X_t)\frac{\partial \sigma(t,X_t)}{\partial x} \, \dif t$ is the It\^o correction term. The fact that the limiting SDE has to be understood in Stratonovich sense is not surprising: if the approximating equations satisfy the usual chain rule, the same should be true for the limiting equation. 
In the multidimensional case, one needs to be careful about how approximating the noise. For instance, P. L\'evy showed that if $(w^{n,(1)}, w^{n,(2)})$ denotes the piecewise linear approximation to a two-dimensional Brownian motion $(W^{(1)},W^{(2)})$, one has
\begin{align*}
    \frac{1}{2} \int_0^t w^{n,(1)}_s\, \dif w^{n,(2)}_s -  w^{n,(2)}_s\, \dif w^{n,(1)}_s \to \frac{1}{2} \int_0^t W^{(1)}_s \circ \dif W^{(2)}_s - W^{(2)}_s \circ \dif W^{(1)}_s.
\end{align*}
The integral on the right hand side is called \emph{L\'evy area}. From this convergence, one can deduce that approximating the Brownian motion piecewise linearly in a multidimensional SDE converges, again, to the Stratonovich version of the SDE. However, E.~J.~McShane \cite{Mcshane} gave an example of an approximation to a two-dimensional Brownian motion for which the limit is \emph{not} the usual L\'evy area, but the area plus another correction term (see Example \ref{examapp} below). This implies that a multidimensional SDE with this approximation will not converge to the usual Stratonovich SDE. More results and surveys on this topic can be found in \cite{INY}, \cite[Chapter 7]{Ikeda},  \cite{GP} and the references therein. The study of Wong-Zakai approximations was further extended to stochastic partial differential equations (abbreviated as SPDEs), too, cf. e.g. \cite{KT}, \cite{MZ}, and \cite{HP}. In \cite{GS} and \cite{GPS}, Gy\"ongy, Shmatkov and Stinga showed that the convergence rate of the Wong-Zakai approximation for SPDEs is $\frac{1}{2}$, which is the same as for SDEs.\\

 Wong-Zakai approximations bridge the gap between the ODE and the SDE world. They are important for many purposes, e.g. in numerics or in subjects concerning the limiting behavior of SDEs. However, all known results exclusively treat SDEs with regular coefficients, i.e. $\sigma$ and $b$ were at least assumed to be Lipschitz continuous. In this article, we are interested in studying the case when the coefficients in the SDE might not be continuous at all but only belong to some $L^p$-space with $p>d$. Then immediately we are in the problematic situation that the approximating equation may not be well-defined for singular coefficients if we simply smooth out the noise. In fact, the roughness of the noise is essential to render the equation well-defined, most visible in the above cited works where a pathwise rough path approach is used. In order to handle this issue, we approximate the singular drift by a smooth one at the same time. Using a stability result that we obtain for singular SDEs (see Theorem \ref{sta} below), we can show that the solution of the approximated equation converges to the original singular SDE in the mean square sense. In the following, we loosely formulate our main result, the precise formulation can be found in Theorem \ref{Wzkthy}.
 
\begin{theorem}\label{thm:intro}
    Let $b$ and $\sigma$ be coefficients defining a singular SDE. In particular, $|b|\in  L^p(\mathbb{R}^d)$, $p\in(d,\infty)$ and $p\geq 2$. Let $(W_t^n)_{n\geq1}$ resp. $(b_n)_{n \geq 1}$ be  suitable approximations for a $d$-dimensional Brownian motion $W_t$ resp. for the coefficient $b$. Assume that $W^n$ converges sufficiently fast to $W$ compared to the convergence of $b_n$ to $b$. Then the solution $(X_t)_{t \geq 0}$ to the singular equation
\begin{align}\label{sdeint}
X_t=x_0+\int_0^tb(X_s) \, \dif s+\int_0^t\sigma(X_s) \, \dif W_s+\Big(\sum_{i,j=1}^d\sum_{l=1}^dc_{ij}\int_0^t\sigma_{il}\cdot\nabla_l \sigma_{jk}(X_s) \, \dif s\Big)_{k=1,\cdots,d}
\end{align}
is the mean square limit of the solution to the equation
\begin{align}\label{odeint}
X_t^{n}=x_0+\int_0^tb_n(X_s^{n})\dif s+\int_0^t\sigma(X_s^{n})\dif W_{s}^n.
\end{align}
Above, $(c_{ij})_{1\leq i,j\leq d}$ are coefficients depending on the approximating sequence $(W_t^n)_{n\geq1}$.
\end{theorem}
The above result shows that certain random, regular ODEs \eqref{odeint} converge to the singular SDE \eqref{sdeint}. The condition that the speed of convergence of the noise approximation has to be faster than for the vector field approximation seems plausible, but it is open for us whether this condition is really necessary.\\

A typical application for the Wong-Zakai theorem is the proof of a support theorem which describes the topological support of the law of an SDE solution \cite{SV, GP, N, BMS}. To our knowledge, there was no support theorem for singular SDEs in the literature until now. In principle, it is possible to mimic the proof of the classical support theorem for regular SDEs using Theorem \ref{thm:intro}. However, it turns out that there is a much simpler proof, using a change of measure method to eliminate the drift which allows to apply a well-established support theorem for non-degenerate diffusions \cite{SV}. Here is a short form of our support theorem, cf. Theorem \ref{suppthy}.

\begin{theorem}
It holds that
$$\textup{supp } \mu=\mathcal{C}([0,T],\mathbb{R}^d),$$
where $\textup{supp }\mu$ is the topological support of the distribution of the solution $(X_t)_{t\geq0}$ to the singular SDE \begin{align*}
X_t=x_0+\int_0^tb(s,X_s)\dif s+\int_0^t\sigma(s,X_s)\circ\dif W_s.
\end{align*} 
\end{theorem}

The structure of this paper is as follows: We introduce some notation and preliminaries in Section 2. In Section 3, we first obtain the stability of the solution to a singular SDE with multiplicative noise with respect to the singular drift by using Zvonkin's transformation. In the second part, we prove the Wong-Zakai Theorem for singular SDEs. The support theorem is shown in Section 4. The appendix collects some necessary results in PDE theory which are used when establishing the stability result in Section 2.

\section{Notation and Preliminaries}
We first introduce some  notation. We write $\mR^+:=[0,\infty)$ and $\mathbb{N}=\{1,2,...\}$.

For $v\in \mR^m$, let $v^{(i)}$, $i=1,\cdots,m$, be the $i$-th component of $v$ and  $|v|$ its $l^2$ Euclidean norm.  For a matrix $\sigma\in\mR^{m\times n}$, we denote its components by $\sigma_{ij}$,  $i=1,\cdots, m$, $j=1,\cdots,n$. We use $\sigma^*$ to denote its transpose and $\Vert\sigma\Vert$ is its Euclidean norm in $\mR^{m\times n}$, i.e.~the root of the sum of the squares of all entries.

For $p$, $q\in[1,\infty)$ and $0\leq S<T\leq \infty$, let $\mathbb{L}_p^q(S,T)$ denote the space of all real Borel measurable functions on $[S,T)\times\mathbb{R}^d$ with the norm
$$\Vert f\Vert_{\mathbb{L}_p^q(S,T)}:=\bigg(\int_S^T\big(\int_{\mathbb{R}^d}|f(t,x)|^p\dif x\big)^{q/p}\, \dif t\bigg)^{1/q}<+\infty.$$
 For simplicity, we write
$$\mathbb{L}_p^q=\mathbb{L}_p^q(\mathbb{R}_+),\quad\mathbb{L}_p^q(T)=\mathbb{L}_p^q(0,T),\quad\mathbb{L}_p^{q,loc}=L^q_{loc}(\mathbb{R}_+,L^p(\mathbb{R}^d)).$$

Let $\mathcal{C}([0,\infty),\mathbb{R}^d)$ denote the space of all continuous $\mathbb{R}^d$-valued functions defined on $[0,\infty)$. By $\mathcal{C}_b^n(\mathbb{R}^d)$ we mean the set of all bounded $n$ times continuously differentiable functions on $\mathbb{R}^d$ with bounded derivatives up to order $n$.  $\mathcal{C}(\mathbb{R}^d)$ collects all continuous functions defined on $\mathbb{R}^d$. For $f\in L_{loc}^1(\mathbb{R}^d)$ we define $\nabla_jf(x):=\frac{\partial f}{\partial x_j}(x)$ and  $\nabla f:=(\nabla_i f)_{1\leq i\leq d}$ denotes the gradient of $f$, $\partial_{ij}^2f(x):=\frac{\partial^2f}{\partial x^{(i)}\partial x^{(j)}}$. Here the derivatives are meant in the sense of distributions. Below when we write $\nabla f$ for a function depending on $(t,x)\in\mathbb{R}^{d+1}$, we always mean $\nabla_xf(t,x)$, i.e.~the derivative with respect to $x\in\mathbb{R}^d$, and $\partial_tf$ denotes the derivative with respect to $t\in\mathbb{R}_+$.
For any $\alpha\in\mR$ and ${ p}\in[1,\infty)$, let $H_{ p}^\alpha(\mathbb{R}^d):=(1-\Delta)^{-\alpha/2}\big(L^{ p}(\mR^d)\big)$
be the usual Bessel potential space with norm $
\|f\|_{H_{ p}^\alpha(\mathbb{R}^d)}:=\|(\mI-\Delta)^{\alpha/2}f\|_{L^{ p}(\mathbb{R}^d)},
$
where  $(\mI-\Delta)^{\alpha/2}f$ is defined through Fourier's  transform.
For  $\alpha\in\mR$, let $
\mH_{\alpha,{ p}}^q(T):=L^q\big([0,T];H^{\alpha}_{ p}(\mR^d)\big),
$
and the space $\sH_{\alpha,{ p}}^q(T)$ consists of the functions $u=u(t)$ on $[0,T]$ with values in the space of distributions on $\mR^d$ such that $u\in \mH_{\alpha,{ p}}^q(T)$ and $\p_t u\in \mL^q_{ p}(T)$.

In the following,  $\int_0^t\sigma(s, X_s)\circ\dif W_s$ is understood as Stratonovich integral, i.e.
$$\int_0^t\sigma(s,X_s)\circ\dif W_s:=\int_0^t\sigma(s,X_s) \, \dif W_s+\frac{1}{2}\int_0^t\sigma(s,X_s)\frac{\partial \sigma(s,X_s)}{\partial x}\, \dif s,$$
where $\int_0^t\sigma(s,X_s)\, \dif W_s$ is the It\^o integral whenever it makes sense.

\section{Wong-Zakai Approximation for SDEs with singular drift driven by multiplicative noise}

The proof of our main result consists of two steps. First, we prove a stability result for singular SDEs that shows that singular SDEs can be approximated by ``classical'' SDEs. Then, we formulate a quantitative version of the Wong-Zakai theorem. Both results together yield our main theorem. We start with step 1.

\subsection{Stability of SDEs with singular drift driven by multiplicative noise}
We present a stability result under the following condition:
\begin{assumption}\label{ass}For $\frac{d}{p}+\frac{2}{q}<1$, $p,q\in[2,\infty)$,
\begin{itemize}
\item[\bf($H_b$)] $b:\mR^+ \times \mR^d\to \mR^d$ is measurable and $|b|\in \mL_p^{q,loc}$.
\item [\bf($H_\sigma$)] $\sigma:\mR^+  \times \mR^d \to \mR^{d \times d}$ is measurable, $|\nabla\sigma|\in \mL_p^{q,loc}$, and there exists $K\geq 1$ such that for all $(t,x)\in[0,\infty)\times\mathbb{R}^d$, $$K^{-1}|\xi|^2\leq(\sigma\sigma^*(t,x)\xi,\xi)\leq K|\xi|^2,\quad \forall \xi\in\mathbb{R}^d.  $$  Further, $a:=\sigma\sigma^*$  is uniformly continuous in $x\in\mathbb{R}^d$ locally uniformly with respect to $t\in[0,\infty)$.
    \end{itemize}
    \end{assumption}
     Concerning existence and uniqueness of a global strong solution   to \eqref{sde00}, \cite{XXZZ} shows the following result.
    \begin{theorem}\cite[Theorem 1.1]{XXZZ} \label{XY}
Suppose  the conditions  in Assumption \ref{ass} hold. Then, for any $x\in\mathbb{R}^d$, there exists a unique strong solution $(X_t(x))_{t\geq0}$ to the SDE \eqref{sde00} such that
\begin{align*}
\mP\left\{\omega:\int_0^T|b(r,X_r(\omega))|\, \dif r+\int_0^T|\sigma(r,X_r(\omega))|^2 \, \dif r<\infty,\forall T\in[0,\infty)\right\}=1,
\end{align*}
and, almost surely,
\begin{align*}
X_t=x+\int_0^tb(r,X_r) \, \dif r+\int_0^t\sigma(r,X_r) \, \dif W_r, \quad \forall t\in [0,\infty).
\end{align*}
\end{theorem}
Next, we formulate the stability result.
\begin{theorem}\label{main1}
  Assume that $b^1,b^2$ satisfy {\bf($H_{b^1}$)}, {\bf($H_{b^2}$)} respectively, and $\sigma$ satisfies {\bf($H_\sigma$)}. Let $(X_t^1(x_1))_{t\geq0}$ and $(X_t^2(x_2))_{t\geq0}$ be the unique strong solutions to the  equations 
\begin{align}\label{sde1}
dX_t^1=b^1(t,X_t^1)\, \dif t + \sigma(t,X_t^1) \, \dif W_t,\quad X_0^1=x_1\in\mathbb{R}^d,\quad t\geq0,
\end{align}
and
\begin{align}\label{sde2}
dX_t^2=b^2(t,X_t^2) \, \dif t + \sigma(t,X_t^2)\, \dif W_t,\quad X_0^2=x_2\in\mathbb{R}^d,\quad t\geq0.
\end{align}
Then, for each $T>0$, 
there exist positive constants $C$  and $C'$ such that
\begin{align}\label{sta}
\mE\(\sup_{s\in[0,T]}|X_s^1(x_1)-X_s^2(x_2)|^2\)\leq C\Vert b^1-b^2\Vert_{\mL_p^q(T)}^2+C'|x_1-x_2|^2
\end{align}
where $C$ and $C'$ may depend on $T,d,p,q,K,\Vert \nabla\sigma\Vert_{\mL_p^q(T)},\Vert b^1\Vert_{\mL_p^q(T)}$, and $\Vert b^2\Vert_{\mL_p^q(T)}$.
\end{theorem}
\begin{proof}
We follow the argument from \cite{Zhang2017}. We show \eqref{sta} via the following two steps.\\

\emph{Step 1. SDEs with different diffusion terms and no drift.}  \\

 Let $\sigma^1$ and $\sigma^2$ satisfy {\bf($H_{\sigma^1}$)}  and   {\bf($H_{\sigma^2}$)} respectively. Let $(Y_t^1(y_1))_{t\geq0}$ and $(Y_t^2(y_2))_{t\geq0}$ be the unique strong solutions to
    \begin{align}\label{Y1}
    dY_t^1 = \sigma^1(t,Y_t^1)\, \dif W_t, \quad Y_0^1=y_1\in\mathbb{R}^d
    \end{align}
    and
    \begin{align}\label{Y2}
    dY_t^2=\sigma^2(t,Y_t^2)\, \dif W_t, \quad Y_0^2=y_2\in\mathbb{R}^d.
    \end{align} 
    By \cite[Theorem 5.3]{Zhang2017}, \cite[Lemma 3.2]{Zhang2011} and the triangle inequality, we get  that there exist  $C_1=C_1(K, d,p,q,\Vert \nabla\sigma^1\Vert_{\mL_p^q(T)},\Vert \nabla\sigma^2\Vert_{\mL_p^q(T)},T)>0$ and $C_2=C_2(K,d,p,q,\Vert \nabla\sigma^2\Vert_{\mL_p^q(T)},T)>0$ such that 
    \begin{align}\label{staY}
    \mE  &\( \sup_{t\in[0,T]}|Y_t^1(y_1)-Y_t^2(y_2)|^2 \) \leq C_1\Vert \sigma^1-\sigma^2\Vert_{\mL_p^q(T)}^2+C_2|y_1-y_2|^2.
    \end{align}
    \\

   \emph{Step 2. SDEs with different drifts and common noise.} 
    Fix $T>0$. For  $k=1,2$ and $l=1,\cdots,d$, {following  from Theorem \ref{pde}}, let $u^{k,(l)}\in \sH_{2,p}^q(T)$ be the unique solution to the equation
$$
\partial_tu^{k,(l)}+\frac{1}{2}\sum_{i,j=1}^da_{ij}\partial_{ij}^2u^{k,(l)}+b^{k,(l)}\cdot\nabla u^{k,(l)}+b^{k,(l)}=0, \quad t\in[0,T], u^{k,(l)}(T,x)=0,l=1,\cdots,d,
$$
where $(a_{ij})_{1\leq i,j\leq d}=\sigma\sigma^*$.
Let $U_{b^k}:=(u^{k,(l)})_{1\leq l\leq d}$, and define 
\begin{align*}
    \Phi^k(t,x):=x+U_{b^k}(t,x) \text{ for } (t,x)\in[0,T]\times\mathbb{R}^{d}.
\end{align*}
 Lemma \ref{hom} shows that there exists some $T>0$ such that  $\Phi^k(t,.)$ is a diffeomorphism for each  $t \in [0,T]$ and that its derivative is jointly continuous in $(t,x) \in [0,T]\times \mR^d$.  Let
 $$Y^1 := \Phi^1(t,X_t^1), \quad \Psi^1(t,y)=({\Phi^1})^{-1}(t,y).$$
 By \cite[Lemma 4.3]{Zhang2011}, we see that $(X_t^1(x_1))_{t\geq0}$ solves the SDE \eqref{sde1} if and only if $(Y_t^1(y_1))_{t\geq0}$ solves \eqref{Y1} with $$\sigma^1(t,y):=[\nabla\Phi^{1}\cdot\sigma]\circ({\Psi^{1}}(t,y)),\quad x_1=\Psi ^1(0,y_1),$$
 and similarly for $(X_t^2(x_2))_{t\geq0}$ and $(Y_t^2(y_2))_{t\geq0}$, i.e., for $i=1,2$, $(X_t^i(x_i))_{t\geq0}:=(\Psi^i(t,Y_t^i(y_i)))_{t\geq0}$ is the solution to \eqref{sde1} and \eqref{sde2} respectively. Notice that by {\eqref{c1hommo} and} Lemma \ref{hom} (2), there exists $N$  such that 
 \begin{align*}|y_1-y_2|&=|\Phi^1(0,x_1)-\Phi^2(0,x_2)|\\&\leq  {\frac{3}{2}}|x_1-x_2|+N\Vert b^1-{b^2}\Vert_{\mL_p^q([s_0,t_0])}.
 \end{align*}
 By \eqref{c1hommo} and Lemma \ref{hom} (1), (2), we have
\begin{align}\label{Psi12}
\sup_{y\in\mathbb{R}^d}|\Psi^{1}(t,y)-\Psi^{2}(t,y)|&=\sup_{y\in\mathbb{R}^d}|y-\Psi^{2}(t,\cdot)\circ\Phi^{1}(t,y)|
\nonumber\\&\leq 2\sup_{y\in\mathbb{R}^d}|\Phi^2(t,y)-\Phi^{1}(t,y)|
\nonumber\\&\leq C_3\Vert b^1-b^2\Vert_{\mL_p^q(T)},
\end{align}
where $C_3=C(K, d,p,q,\Vert b^1\Vert_{\mL_p^q([s_0,t_0])},\Vert {b^2}\Vert_{\mL_p^q([s_0,t_0])})$.
 Then, combining \eqref{Psi12} with \eqref{staY} and Lemma \ref{hom} (1), we get
\begin{align}\label{Zvonkxy}
\mE&\sup_{t\in[0,T]}|X_t^1(x_1)-X_t^2(x_2)|^2
\nonumber\\&=\mE\sup_{t\in[0,T]}|\Psi^{1}(t,Y_t^1(y_1))-\Psi^{2}(t,Y_t^2(y_2))|^2\nonumber\\&\leq 2\mE\sup_{t\in[0,T]}|\Psi^{1}(t,Y_t^1(y_1))-\Psi^{1}(t,Y_t^2(y_2))|^2+2\mE\sup_{t\in[0,T]}|\Psi^{1}(t,Y_t^2(y_2))-\Psi^{2}(t,Y_t^2(y_2))|^2
\nonumber\\&\leq 2\Vert \nabla \Psi^1\Vert_{\mL^\infty(T)}^2 \mE\sup_{t\in[0,T]}|Y_t^1(y_1)-Y_t^2(y_2)|^2+C_4\Vert b^1-b^2\Vert_{\mL_p^q(T)}^2
\nonumber\\&\leq C_5\Vert b^1-b^2\Vert_{\mL_p^q(T)}^2+C_6|y_1-y_2|^2
\nonumber\\&\leq C_5\Vert b^1-b^2\Vert_{\mL_p^q(T)}^2+C_7|x_1-x_2|^2,
\end{align}
where all constants appearing are positive. 
 Therefore, we obtain \eqref{sta} for $T=T_0$ sufficiently small. For general $T$, we iterate the argument (as in  \cite[Theorem 1.1]{Zhang2017}) and show that \eqref{sta} holds on the interval $[\frac{T_0}{2}, \frac{3T_0}{2}]$ etc. By the pathwise uniqueness property, we can patch up the solution and conclude the proof.
\end{proof}
\begin{remark}
 In \cite[Theorem 1.1 E]{Zhang2017} and \cite[Theorem 3.10]{XZ2}, similar stability results were obtained under slightly stronger assumptions.  In \cite[Theorem 6.1]{Zhang2017}, the authors assumed in addition  that  $\nabla\sigma, b^i\in \mathbb{L}^q_p(T)$ with $p=q$ and $p>d+2$, $i=1,2$, to show well-posedness of the corresponding Kolmogorov equation and, consequently, the stability result.  
 In \cite[Theorem 3.10]{XZ2}, the condition  $\nabla b^1\in \mathbb{L}^{q'}_{p'}(T)$ for some $p',q'\in(1,\infty)$ was assumed to obtain the stability result.
\end{remark}

\subsection{Wong-Zakai approximations}
Next, we prove a quantitative version of the classical Wong-Zakai theorem. We follow the exposition in \cite{Ikeda}. 

Let $(\mathbb{W},\mP,\mathcal{F},(\mathcal{F}_t)_{t\geq0})$ be a Wiener space, i.e. $\mathbb{W}:=\{\omega\in\mathcal{C}([0,\infty),\mathbb{R}^d):\omega(0)=0\}$, $\mP$ is the Wiener measure defined on the Borel $\sigma$-algebra $\mathcal{F}$ of $\mathbb{W}$, $W_t(\omega):=\omega(t),\,t \geq 0$ is a Wiener process and $(\mathcal{F}_t)_{t\geq0}$ is the filtration generated by  $(W_t)_{t \geq 0}$. We shall consider the following class of approximations:
\begin{definition}\cite[VI. Definition 7.1]{Ikeda}\label{apw}
By an approximation of the Wiener process $W_t(\omega):=\omega(t),\,t \geq 0$, we mean a family $\{W^n\}_{n\geq0}$ of $d$-dimensional continuous processes defined on the Wiener space $(\mathbb{W},\mP,\mathcal{F},(\mathcal{F}_t)_{t\geq0})$ such that
\begin{itemize}
    \item[(1)] for every $\omega\in\mathbb{W}, t\mapsto W^n_t(\omega)$ is continuous and piecewise continuously differentiable,
    \item[(2)] $W^n_0(\omega)$ is $\mathcal{F}_\frac{1}{n}$-measurable, and $\mE W^{n}_0=0$,
    \item[(3)] $W^n_{t+\frac{k}{n}}(\omega)=W^n_{t}(\theta_{\frac{k}{n}}\omega)+\omega(\frac{k}{n})$, for every $k=1,2,\cdots, t\geq0$ and $\omega$, where $\theta_t,t\geq0$ is the shift operator defined by $(\theta_t\omega)(s)=\omega(t+s)-\omega(t)$ for $\omega\in\mathbb{W}$,
    \item[(4)] there exists a positive constant $C$ such that $$\mE[|W^n_{\frac{1}{n}}|^6]\leq C\frac{1}{n^3},\quad \mE\(\int_0^{\frac{1}{n}}|\frac{\dif }{\dif s}W^n_s|\dif s\)^6\leq C\frac{1}{n^3}.$$
\end{itemize}
\end{definition}

From \cite[VI. Section 7]{Ikeda} we know that for such a sequence $W^n$ we have 
$$\lim_{n\rightarrow\infty}\mE\sup_{0\leq t\leq T}|W_t-W_t^n|^2=0$$
for every $T>0$.
We further introduce the following notation for $t>0$: 
\begin{align}\label{sij}
S_{ij}(t):= \frac{\(\int_0^t W_s^i\circ \dif W^j_s-W_s^j\circ \dif W^i_s\)}{{{2}}t},\quad i,j=1,\cdots,d,\quad,t>0.
\end{align}
Further,
\begin{align}\label{sijn}
s_{ij}^n(t):=s_{ij}(t,n):=\mE S_{ij}(t,n)=\frac{\mE\(\int_0^tW^{n,i}_s\frac{\dif }{\dif s}W^{n,j}_s-W^{n,j}_s\frac{\dif }{\dif s}W^{n,i}_s\dif s\)}{2t},\quad i,j=1,\cdots,d,
\end{align}
and
\begin{align}\label{cijn}
c_{ij}^n(t):=c_{ij}(t,n):=\frac{\mE\Big[\int_0^t\frac{\dif }{\dif s}W^{n,i}_s\(W^{n,j}_t-W^{n,j}_s\)\dif s\Big]}{t}
,\quad i,j=1,\cdots,d.
\end{align}
Further, $S_{ij}(0):=0$, $s_{ij}^n(0):=s_{ij}(0):=0$ and $c_{ij}^n(0):=c_{ij}(0):=0$.

Observe that $(s_{ij}(t,n)):=(s_{ij}(t,n))_{1\leq i,j\leq d}$ is a skew-symmetric $d\times d$-matrix for each $t$ and $n$, i.e. $s_{ij}(t,n)=-s_{ji}(t,n)$.\\

 The process $S_{ij}(t)$ is known as the \emph{L\'evy area} in the literature and plays a fundamental role in rough paths theory. As explained in \cite{INY}, for such a general approximation $\{W^n_t\}$ of $W_t$, $t\geq0$, there are cases in which $(S_{ij}(t,n))$ converges to $(S_{ij}(t))$ (e.g.\cite{Levy}) and others for which $(S_{ij}(t,n))$ does not converge to $S_{ij}(t)$ as $n\rightarrow\infty$ but $S_{ij}(t)+$ 'another correction term' does \cite{Mcshane}. Here we consider a class of approximations of the Wiener process including both cases by assuming the following condition:
 \begin{assumption}\label{asssc}
 There exists a skew-symmetric $d\times d$-matrix $(s_{ij})_{1\leq i,j\leq d}$ and a rate function $f_n$ such that
 \begin{align*}
 \big|s_{ij}^n(\frac{1}{n})-s_{ij}\big|\leq f_n,\quad \lim_{n\rightarrow\infty}f_n=0,\quad i,j=1,\ldots,d.
 \end{align*}
 \end{assumption}
 The following proposition follows from the proof of \cite[VI. Lemma 7.1]{Ikeda} in a straightforward manner.
 \begin{proposition}Define
 \begin{align}\label{cij}
c_{ij}:=s_{ij}+\frac{1}{2}\delta_{ij},\quad i,j=1,\ldots,d,
\end{align}
and let $Z:(0,1]\to \mathbb{N}$ be a function such that $\lim_{\delta\rightarrow0}Z(\delta)=\infty$.
  Under Assumption \ref{asssc}, there exists a positive constant $C$, independent of $n$, such that
 \begin{align}\label{cijncij}
 \Big|c_{ij}\Big(\frac 1n Z\big(\frac{1}{n}\big),n\Big)-c_{ij}\Big|&\leq  \Big|s_{ij}^n(\frac{1}{n})-s_{ij}\Big|+ C\Big(Z(\frac{1}{n})^{-1}+Z(\frac{1}{n})^{-1/2}\Big)\nonumber\\&\leq f_n+2CZ(\frac{1}{n})^{-1/2}.
 \end{align}
 In particular, $$\lim_{n\rightarrow\infty}c_{ij}\Big(\frac 1n Z\big(\frac{1}{n}\big),{n}\Big)=c_{ij}.$$
 \end{proposition}
  We provide  three classical examples for such approximations. For this purpose, we introduce the following notation. Denote
$$\mathcal{D}:=\{f: [0,1]\rightarrow \mathbb{R} \text { continuously differentiable},\quad f(0)=0 \text{ and } f(1)=1.\}$$
\begin{example}\label{examapp}
\begin{itemize}
\item[(1)] For $f\in\mathcal{D}$, let
       $$W_t^n:=W_{\frac{k}{n}}+f(t-\frac{k}{n})(W_{\frac{k+1}{n}}-W_{\frac{k}{n}}), \quad t\in[\frac{k}{n},\frac{k+1}{n}),\quad k=1,2,\cdots.$$
       Then according to \cite[VI. Example 7.1]{Ikeda} such $(W^n_t)_{n\geq1}$ satisfies Definition \ref{apw}, hence it is an approximation of $W_t$, $t\geq0$. In this case, we easily observe that
       $$s_{ij}^n(\frac{1}{n})=0,\quad s_{ij}=0,\quad c_{ij}=\frac{1}{2}\delta_{ij},\quad i,j=1,\cdots,d.$$
       It implies that Assumption \ref{asssc} is fulfilled by taking $f_n=0$ for $n\in\mathbb{N}_+$.

       In particular if we take $f(t)=t$, then $(W^n)_{n\geq1}$ is the familiar piecewise linear approximation.
       \item[(2)] Mollification: Let $\rho$ be a non-negative $\mathcal{C}^\infty$-function with support contained in $[0,1]$ and $\int_0^1\rho(s)\dif s=1$. Set $$\rho_n(s)=n\rho(ns),\quad n\geq0$$
           and
           $$W_s^n=W_\cdot\ast\rho_n(s)=\int_0^\infty W_r\rho_n(s-r) \, \dif r.$$
           We can verify that $(W^n)_{n\geq 1}$ is an approximation of a Wiener process (see, e.g. \cite[VI. Example 7.3]{Ikeda}). Also in this case
           $$s_{ij}^n(\frac{1}{n})=0,\quad s_{ij}=0,\quad i,j=1,\cdots,d.$$
            Hence Assumption \ref{asssc} is satisfied by taking $f_n=0$ and $c_{ij}=\frac{1}{2}\delta_{ij}$.
           \item[(3)] (E.J. McShane\cite{Mcshane}) Let $d=2$, $f^i\in\mathcal{D}$, $i=1,2$. For $i=1,2$, we define for
           $t\in[\frac{k}{n},\frac{k+1}{n}),$ where $n\in \mathbb{N}_+$, $k=0,1,\cdots,n-1,$
           \begin{equation*}W^{n,i}_t=
     \left\{
     \begin{aligned}
     & W^i_{\frac{k}{n}}+f^i(n(t-\frac{k}{n}))(W_{\frac{k+1}{n}}^i-W_{\frac{k}{n}}^i), \text{ if } (W_{\frac{k+1}{n}}^1-W_{\frac{k}{n}}^1)(W_{\frac{k+1}{n}}^2-W_{\frac{k}{n}}^2)\geq0,\\
    &W^i_{\frac{k}{n}}+f^{3-i}(n(t-\frac{k}{n}))(W_{\frac{k+1}{n}}^i-W_{\frac{k}{n}}^i), \text{ otherwise. } \\
    \end{aligned}
    \right.
    \end{equation*}
          $(W^n)_{n\geq 1}$ fulfills all conditions in Definition \ref{apw} (e.g. \cite[VI. Example 7.2]{Ikeda}). In this case, we can verify that 
          $$s_{11}=s_{22}=0,\quad s_{12}=\frac{1}{\pi}, \quad s_{21}=-\frac{1}{\pi}.$$
\end{itemize}
\end{example}


In the rest of this section, we only consider time homogeneous coefficients, 
i.e.~ $b:\mathbb{R}^d\rightarrow\mathbb{R}^d$ and $\sigma:\mathbb{R}^d\rightarrow\mathbb{R}^{d\times d}$. 
We believe that the extension to inhomogeneous coefficients is straightforward, but we restrict ourselves to this case for the sake of simplicity.

To prove our main result, we need a quantitative version of the classical Wong-Zakai theorem, i.e. an upper bound for the convergence rate. It turns out that a careful inspection of the proof of \cite[VI. Theorem 7.2]{Ikeda} even leads to the known optimal convergence rate. We state this result in the next theorem.

\begin{theorem}\label{smoothdrift}
Suppose that for $1\leq i,j\leq d$, $\sigma_{ij}\in\mathcal{C}_b^2(\mathbb{R}^d)$, $b^{(i)}\in \mathcal{C}_b^1(\mathbb{R}^d)$. Let $(W^n_t)_{n\geq 1}$ be a sequence of approximations for $W_t,t\geq0$ satisfying all conditions in Definition \ref{apw}. Assume  that $(W_t^n)_{n\geq 1},t\geq0,$ satisfies  Assumption \ref{asssc}. Consider the following equations
\begin{align}\label{eqn}
X_t=x_0+\int_0^tb(X_s) \, \dif s + \int_0^t\sigma(X_s)\, \dif W_s + \Big(\sum_{i,j=1}^d\sum_{l=1}^dc_{ij}\int_0^t\sigma_{il}\cdot\nabla_l \sigma_{jk}(X_s)\, \dif s\Big)_{k=1,\cdots,d},  
\end{align}
and
\begin{align*}
X_t^{n} = x_0 + \int_0^tb(X_s^{n})\, \dif s + \int_0^t\sigma(X_s^{n})\, \dif W_{s}^n,\quad t \geq 0.
\end{align*}
Then, for each $T>0$ and $\delta\in(0,1)$,
\begin{align}\label{errorest}
\mE\(\sup_{t\in[0,T]}|X_t-X_t^{n}|^2\)\leq C_1e^{C_2\Vert b\Vert_{\mathcal{C}_b^1(\mathbb{R}^d)}^2}(f^2_n+(1+\sup_{x\in\mathbb{R}^d}|b(x)|^2)\frac{1}{n^{1-\delta}}),
\end{align}
where the constants $C_1,C_2$ depend on $\delta$, $\sigma$ and $T$ but not on  $b$ and $n$, and $f_n$ is from Assumption \ref{asssc}.
\end{theorem}
\begin{proof} We closely follow the argument  in the proof of \cite[VI. Theorem 7.2]{Ikeda} but we keep track of  the convergence rate of the approximating sequence $\{X_t^{n}\}_{n\geq 1}$  to $X_t$ in the mean square sense on $t\in[0,T]$.

 For $\delta\in(0,1)$, let {{$Z_\delta(s):=\lfloor s^{-\frac{\delta}{4}}\rfloor$, $s \in (0,1)$.}} Notice that such functions $Z_\delta$ satisfy
 \begin{align}\label{Zdelta}
     Z_\delta:(0,1)\rightarrow\mathbb{N}, \quad \lim_{n\rightarrow\infty}Z_\delta\Big(\frac{1}{n}\Big)^4n^{-1}\leq \lim_{n\rightarrow\infty} \frac{1}{n^{1-\delta}}=0,\quad \lim_{s\rightarrow 0} Z_\delta(s)=\infty.
 \end{align}
Therefore for such $Z_\delta$, by \eqref{cijncij}, we have
$$c_n:=\Big|c_{ij}\Big(\frac{1}{n}Z_\delta\Big(\frac{1}{n}\Big),\frac{1}{n}\Big)-c_{ij}\Big|\leq f_n+C Z_\delta\Big(\frac{1}{n}\Big)^{-1/2}.$$
Then, as in \cite[VI. (7.74)]{Ikeda}, there exist constants $C'(\sigma),C''(\sigma)>0$ independent of $b$ such that
\begin{align*}
\mE&\(\sup_{t\in[0,T]}|X_t-X_t^{n}|^2\)
\\&\leq C'\Vert b\Vert_{\mathcal{C}_b^1(\mathbb{R}^d)}^2\int_0^T \mE\(|X_t-X_t^{n}|^2\)\dif t+C''[(1+\sup_{x\in\mathbb{R}^d}|b(x)|^2)Z_\delta(n^{-1})^4n^{-1}+c_n^2].
\end{align*}
Finally, Gronwall's inequality together with \eqref{Zdelta} show that \eqref{errorest} holds.
\end{proof}

\begin{remark}
    If $f_n = 0$ for every $n$, Theorem \ref{smoothdrift} shows that the Wong-Zakai approximation converges with rate arbitrarily close to $\frac{1}{2}$ uniformly in $L^2(\mathbb{P})$. This is in line with the known optimal (strong) convergence rate for the Wong-Zakai approximation.
\end{remark}

As we already mentioned earlier, to formulate a Wong-Zakai theorem for singular SDEs, we cannot just simply replace the noise by a smooth approximation since this will render the equation ill-posed. To solve this problem, we will approximate the singular coefficient $b$ at the same time. We will now introduce a class of approximations for this $b$.

\begin{definition}[$\mathcal{A}(b,p,h)$: approximation of $|b|\in L^p(\mathbb{R}^d)$]\label{bph} Let $(b_n)_{n \geq 1}$ be a sequence of functions $b_n \colon \mathbb{R}^d \to \mathbb{R}^d$. For a function $h\in[0,\infty)\rightarrow[0,\infty)$, we write $(b_n)_{n\geq 1}\in \mathcal{A}(b,p,h)$ if
\begin{itemize}
\item[(i)] $b_n^{(i)}\in \mathcal{C}_b^1(\mathbb{R}^d)$, $i = 1,\ldots,d$, 
\item[(ii)]  $\Vert b-b_n\Vert_{L^{p}(\mathbb{R}^d)}\rightarrow0$ as $n\rightarrow\infty$,
\item[(iii)] $\Vert b_n\Vert_{\mathcal{C}_b^1(\mathbb{R}^d)}\leq h(n)\Vert b\Vert_{L^p(\mathbb{R}^d)}$.
\end{itemize}
\end{definition}
Next, we state our main theorem.

\begin{theorem}\label{Wzkthy}
Assume that $b \colon \mathbb{R}^d \to \mathbb{R}^d$ is measurable and that $|b| \in L^p(\mathbb{R}^d)$ for some $p \in (d,\infty)$ and $p\geq 2$. Let $\sigma$ satisfy condition $(H_\sigma)$ in Assumption \ref{ass} and for $1\leq i,j\leq d$, $\sigma_{ij}\in\mathcal{C}_b^2(\mathbb{R}^d)$.
Assume further that $(W^n_t)$ is sequence of approximations for $W_t,t\geq0$, satisfying all the conditions in Definition \ref{apw} and Assumption \ref{asssc}.  Let $(X_t)_{t\geq0}$ be the solution to 
\begin{align}\label{wzksds}
X_t=x_0+\int_0^tb(X_s) \, \dif s+\int_0^t\sigma(X_s) \, \dif W_s + \Big(\sum_{i,j=1}^d\sum_{l=1}^dc_{ij}\int_0^t\sigma_{il}\cdot\nabla_l \sigma_{jk}(X_s) \, \dif s\Big)_{k=1,\cdots,d}
\end{align}
with $c_{ij}$ defined as in Assumption \ref{asssc}. Let $(b_n)_{n\geq 1}\in \mathcal{A}(b,p,h)$ be a smooth approximation sequence of $b$ such that
 \begin{align}\label{hfn}
 \lim_{n\rightarrow\infty} e^{h(n)^2\Vert b\Vert_{L^p(\mathbb{R}^d)}^2}(f^2_n+(1+h(n)^2\Vert b\Vert_{L^p(\mathbb{R}^d)}^2)\frac{1}{n^{1-\delta}})=0,\quad \text{for some }\delta\in(0,1),
 \end{align}
 where $f_n$ is given in Assumption \ref{asssc}. Then for every $T>0$ we have \begin{align}\label{error}
\lim_{n\rightarrow\infty} \mE\Big[\sup_{0\leq t\leq T}|X_t-X_t^{n}|^2\Big]=0.
\end{align}
Here $(X_t^{n})_{t\geq0}$ is the solution to the equation
\begin{align*}
X_t^{n}=x_0+\int_0^tb_n(X_s^{n}) \, \dif s + \int_0^t\sigma(X_s^{n}) \, \dif W_{s}^n.
\end{align*}
\end{theorem}

\begin{proof}
First  under Assumption \ref{ass} we know that $|((\sum_{i,j=1}^d\sum_{l=1}^d\sigma_{il}\cdot\nabla_l\sigma_{ik})^{(k)})_{k=1,\cdots,d}|\in {L}^p(\mathbb{R}^d)$, therefore by Theorem \ref{XY} there exists a unique solution $(X_t(x))_{t\geq0}$ to the SDE \eqref{wzksds}. For $(b_m)_{m\in\mathbb{N}}\in\mathcal{A}(b,p,h)$, 
let $(X_{m,t})_{t\geq 0}$ be the solution to
\begin{align*}
X_{m,t}=x_0+\int_0^tb_m(X_{m,s})\dif s&+\int_0^t\sigma(X_{m,s})\dif W_s
\\&+\Big(\sum_{i,j=1}^d\sum_{l=1}^dc_{ij}\int_0^t\sigma_{il}\cdot\nabla_l \sigma_{jk}(X_{m,s})\dif s\Big)_{k=1,\cdots,d}.
\end{align*}
By Theorem \ref{smoothdrift}, we know that for any $m\geq1$,
\begin{align*}
 \mE\Big[\sup_{0\leq t\leq T}|X_{m,t}-X_{m,t}^{n}|^2\Big]\leq& C''e^{C'\Vert b_m\Vert_{\mathcal{C}_b^1(\mathbb{R}^d)}^2}(f^2_n+(1+\Vert b_m\Vert_{\mathcal{C}_b^1(\mathbb{R}^d)}^2)\frac{1}{n^{1-\delta}})
  \\\leq& C''e^{C'h(m)^2\Vert b\Vert_{L^p(\mathbb{R}^d)}^2}(f^2_n+(1+h(m)^2\Vert b\Vert_{L^p(\mathbb{R}^d)}^2)\frac{1}{n^{1-\delta}}),
\end{align*}
where $C'$ and $C''$  are positive constants, independent of $m$ and $n$, and $(X_{m,t}^{n})_{t\geq0}$ solves
\begin{align*}
X_{m,t}^{n}=x_0+\int_0^tb_m(X_{m,s}^{n})\dif s+\int_0^t\sigma(X_{m,s}^{n})\dif W_{s}^n.
\end{align*}
  By choosing $m=n$  and using the condition \eqref{hfn} we can obtain that
\begin{align*}
\lim_{n\rightarrow\infty} \mE\Big[\sup_{0\leq t\leq T}|X_{n,t}-X_{n,t}^{n}|^2\Big]=0.
\end{align*}
Then by Theorem \ref{main1},
\begin{align}\label{diff}
\mE\Big[\sup_{0\leq t\leq T}|X_t-X_t^{n}|^2\Big]&\leq \mE\Big[\sup_{0\leq t\leq T}|X_t-X_{n,t}|^2\Big]+\mE\Big[\sup_{0\leq t\leq T}|X_{n,t}-X_{n,t}^{n}|^2\Big]\nonumber\\
&\leq C^*(p,d,T,K,\Vert \nabla\sigma\Vert_{L_p(\mathbb{R}^d)},\Vert b\Vert_{L_p(\mathbb{R}^d)})\Vert b-b_n\Vert_{L^p(\mathbb{R}^d)}^2\nonumber\\&\quad\quad+\mE\Big[\sup_{0\leq t\leq T}|X_{n,t}-X_{n,t}^{n}|^2\Big].
\end{align}
Therefore \eqref{error} holds by letting $n\rightarrow\infty$. 
\end{proof}
\begin{remark}
\begin{enumerate}
    \item Theorem \ref{Wzkthy} states that an approximation to a singular SDE converges to this SDE provided that the approximation of the noise converges sufficiently fast compared to the approximation of the singular coefficient $b$, cf. condition \eqref{hfn}. Although this condition seems plausible and is true in the ``extreme'' case (one can not first let $b_n \to b$ and then $W^n \to W$), we do not know at the moment whether such a condition is really necessary. We leave this point as an open problem.
    \item It is easy to deduce a convergence rate of the Wong-Zakai approximation. However, this can arbitrarily slow, depending on the convergence rate of $b_n \to b$. For a smooth $b$, we rediscover the known convergence rate of (almost) $\frac{1}{2}$ in the classical Wong-Zakai theorem.
    \item The approximation class $\mathcal{A}(b,p,h)$ turns out to be of own interest and is discussed in more detail in \cite{KLY}.
\end{enumerate}

\end{remark}


In the end we give some examples to show how this theorem works in different cases.
\begin{example}\rm{
\begin{itemize}
\item[(1)] If for $i=1,\cdots,d$, $b^{(i)}\in\mathcal{C}_b^1(\mathbb{R}^d)$ and $|b|\in L^p(\mathbb{R}^d)$, $p>d$ and $p\geq4$, then we can simply take $b_n\equiv b$ and choose $h(n)\equiv1$. This is the special setting  of the standard Wong-Zakai theorem for regular SDEs, e.g. see \cite[VI. Section 7]{Ikeda}.
    \item[(2)] Let $d=1$ and  define
    $$b(x)=1_{[0,1]}(x), \quad x\in\mathbb{R}.$$ Then $|b|\in L^p(\mathbb{R})$ for $p\in[1,\infty]$.
 For $\chi:[0,\infty)\rightarrow(0,\infty)$ satisfying $\lim_{n\rightarrow\infty}\chi(n)=\infty$, let
    \begin{equation*}b_n(x)=
     \left\{
     \begin{aligned}
      0, &\text{ if } x\in(-\infty,-\frac{2}{\chi(n)})\cup(1+\frac{2}{\chi(n)},\infty),\\
    \frac{\chi(n)x}{2}+1, &\text{ if } x\in[-\frac{2}{\chi(n)},0), \\
     -\frac{\chi(n)x}{2}+\frac{\chi(n)+2}{2}, &\text{ if } x\in(1,1+\frac{2}{\chi(n)}], \\
     1,&\text{ if } x\in[0,1].
    \end{aligned}
    \right.
    \end{equation*}
    We find that $b_n\in\mathcal{C}_b^1(\mathbb{R})$ almost everywhere.\footnote{Theorem \ref{Wzkthy} assumes $b_n$ bounded and differentiable \emph{everywhere}, but it is easy to see that $b_n$ can be replaced by smooth versions for the price of adding arbitrary small $\epsilon$ in the following bounds. This will not affect the overall conclusion.} Furthermore,
   $$
    \(\int_{-\infty}^{+\infty}|b_n(x)-b(x)|^p\dif x\)^{1/p}=2\(\frac{2}{\chi(n)(p+1)}\)^{1/p},
    $$
    and
    $$\sup_{x\in\mathbb{R}}|\frac{ \dif b_n(x)}{\dif x}|+\sup_{x\in\mathbb{R}}|b_n(x)|\leq \frac{\chi(n)+2}{2}\Vert b\Vert_{L^p(\mathbb{R})}=\frac{\chi(n)+2}{2}.$$
   Now we take
    $$\chi(n):=\max({2\sqrt{|\log(n^\alpha)|}}-2,0), $$
    $\alpha\in(0,1)$, and choose $\alpha$ and $\delta$ such that $\lim_{n\rightarrow\infty} n^\alpha (f_n^2+\frac{(1+\log n^\alpha)}{n^{1-\delta}})=0$, assuming that this is possible. Here $f_n$ and $\delta$ keep the same meaning as in Theorem \ref{Wzkthy}. For instance, if $f_n = 0$, given any $\alpha \in (0,1)$, we can choose $\delta$ such that $\alpha + \delta < 1$.  For  $h(n) := \frac{\chi(n)+2}{2}$, this implies
    $$\lim_{n\rightarrow\infty} e^{h(n)^2\Vert b\Vert_{L^p(\mathbb{R}^d)}^2}(f^2_n+(1+h(n)^2\Vert b\Vert_{L^p(\mathbb{R}^d)}^2)\frac{1}{n^{1-\delta}})=0. $$
    Therefore $(b_n)_{n\geq1}\in\mathcal{A}(b,p,h)$ and fulfills all conditions in Theorem  \ref{Wzkthy}. We can apply this approximation sequence to the SDE with singular drift $b$ and obtain the Wong-Zakai approximation result.

    Actually, we can also take a mollification sequence of $b$ by convoluting with a smooth function. That is to say, let
    $$b_n(x):=b\ast g_n(x):=\sqrt{\frac{\kappa_n}{2\pi}}\int_{-\infty}^{+\infty} b(x-y)e^{\frac{-\kappa_ny^2}{2}}\dif y,$$
where $\kappa_n>0$  we will choose later.
Notice that
such $b_n$ is smooth with bounded derivatives, and there exists a positive constant $C>0$, independent of $n$, such that
 $$\sup_{x\in\mathbb{R}}|\frac{ \dif b_n(x)}{\dif x}|+\sup_{x\in\mathbb{R}}|b_n(x)|\leq C(\kappa_n+1).$$
 By taking
 $$\kappa_n:=\max(\frac{\sqrt{|\log(n^\alpha)|}}{C}-1,0), $$
 with $\alpha\in(0,1)$ such that $\lim_{n\rightarrow\infty} n^\alpha (f_n^2+\frac{1+\log n^\alpha}{n^{1-\delta}})=0$
and
 $h_n:=C(\kappa_n+1)$, we see that \eqref{hfn} holds.
 As before, we know that such approximation can be applied in order to establish the Wong-Zakai approximation to the singular SDE.
\end{itemize}}
\end{example}

\section{Support theorem for Singular SDEs}
Let $(\Omega,\mathcal{F},(\mathcal{F}_t)_{t\geq0},\mP)$ be a probability spaces and $(W_t)_{t\geq0}$ a Wiener process in $\mathbb{R}^d$. 
 Consider the SDE
\begin{align}\label{Wor}
X_t = x_0 + \int_0^tb(s,X_s) \, \dif s + \int_0^t\sigma(s,X_s)\circ\dif W_s,\quad t \in [0,T].
\end{align}
We are interested in describing the topological support of the distribution of the solution $(X_t)_{t \in [0,T]}$ in the space $\mathcal{C}([0,T],\mathbb{R}^d)$ equipped with the topology of uniform convergence when $b$ is a singular drift.
The classical strategy to derive a description of the support uses the Wong-Zakai theorem, { see e.g. \cite{Ikeda}, \cite{GP}, \cite{MZ}.}   Here, we use a different approach which turns out to be much easier in our context. Namely, we use Girsanov theory to show that actually the support of the distribution of the solution to \eqref{Wor} is the full canonical space $\mathcal{C}([0,T],\mathbb{R}^d)$, which essentially is because of the non degenerated noise.

\begin{theorem}\label{suppthy}
Assume that the conditions  in Assumption \ref{ass} are satisfied. Let
$$\Omega_{x_0}([0,T],\mathbb{R}^d):=\{\omega\in\mathcal{C}([0,T],\mathbb{R}^d):w_0=x_0\}.$$
 Then
$$\textup{supp } \mu=\Omega_{x_0}([0,T],\mathbb{R}^d),$$
where $\textup{supp }\mu$ is the topological support of the distribution of the solution $(X_t(x_0))_{t \in [0,T]}$ to the SDE \eqref{Wor} in $\mathcal{C}([0,T],\mathbb{R}^d)$.
\end{theorem}
\begin{proof}
First we consider the following equation
\begin{align*}
Y_t=x_0+\int_0^t\sigma(s,Y_s)\circ\dif W_s,\quad t\geq0.
\end{align*}
By Theorem \ref{XY}, there is an unique and strong solution $(Y_t(x_0))_{t\geq0}$ to the above SDE.

Let $\textup{supp }\mu_Y$ be the topological support of the distribution of  $(Y_t(x_0))_{t\geq0}$  in $\mathcal{C}([0,T],\mathbb{R}^d)$.
Under the condition $(H_\sigma)$, by \cite[Theorem 3.1]{SV} we know that
$$\textup{supp } \mu_Y=\Omega_{x_0}([0,T],\mathbb{R}^d).$$
 Besides, because of \cite[Lemma 4.1 (ii)]{XXZZ},  for  any constant $\kappa>0$, for any $0\leq T<\infty$, we have
	\begin{align}\label{ex}
	\mE\exp\bigg\{\kappa\int_0^T|b(s,Y_s)|^2\dif s\bigg\}\leq C(T,\kappa,d,{ p},q,\|b\|_{\mathbb{L}_p^q(T)})<\infty,
	\end{align}
i.e. Novikov's condition is fulfilled. 	As a result, we have
	\begin{align*}
	\mE\rho_T:=\mE\exp\bigg\{{\int_0^T\big[b^*\sigma^{-1}\big](s,Y_s)\dif W_s}-\frac{1}{2}\int_0^T\big[b^*(\sigma\sigma^*)^{-1}b\big](s,Y_s)\dif s\bigg\}=1.
	\end{align*}
Then by the  Girsanov theorem, for any bounded functional $f$ defined on $\mathcal{C}([0,T],\mathbb{R}^d)$,
$$\mE[f(X_\cdot)]=\mE[\rho_Tf(Y_\cdot)], \quad 0\leq T<\infty.$$
This shows that the distribution of the solution $(X_t)_{t\geq0}$ to SDE \eqref{Wor} in $\mathcal{C}([0,T],\mathbb{R}^d)$ is absolutely continuous with respect to the distribution of  $(Y_t)_{t\geq0}$  in $\mathcal{C}([0,T],\mathbb{R}^d)$ {{and vice versa}}, therefore they share the same support, which completes our proof.
\end{proof}

\renewcommand\thesection{A}
\section{Auxiliary results on Kolmogorov backward equations}

In this part, we collect the known results on regularity estimates for corresponding PDEs. 
 \cite[Theorem 3.2]{XXZZ} and \cite[Theorem 2.1]{LX} show that the   Kolmogorov backward equation corresponding to the solution to SDE \eqref{sde00} can be solved in $\mathbb{L}_p^q(T)$ under the following assumption.
 \begin{assumption}\label{Aa}
  $(a_{ij})_{1\leq i,j\leq d}$ is uniformly continuous in $x\in\mathbb{R}^d$ locally uniformly with respect to $t\in\mathbb{R}_+$, and for some $\delta\geq 1$ and for $(t,x)\in[0,\infty)\times\mathbb{R}^d$,
  $$\delta^{-1}|\xi|^2\leq |a(t,x)\xi|^2\leq \delta|\xi|^2,\quad \forall \xi\in\mathbb{R}^d.$$
  \end{assumption}

\begin{theorem}\cite[Theorem 2.1]{LX}\label{pde}
Let Assumption \ref{Aa} be satisfied, and let  $f\in{{\mathbb{L}}}_p^{q}(T)$ with $T>0$, $p,q\in(1,\infty)$ and $d/p+2/q<1$. Assume that $b$ is measurable and that $|b|\in{{\mathbb{L}}}_{p_1}^{q_1}(T)$ with $p_1\in[p,\infty)$, $q_1\in[q,\infty)$ such that  $d/p_1+2/q_1<1$. Then in ${\sH}_{2,p}^{q}(T)$, there is a unique solution to the equation \begin{align}\label{pdelambda}
\partial_tu+\frac{1}{2}a_{ij}\partial_{ij}^2u+b\cdot\nabla u+f=0, \quad t\in[0,T],\quad u(T,x)=0,
\end{align}
which further satisfies the following regularity estimates:
for any $\alpha\in [0,2-\tfrac{2}{q})$, there exists a  constant $C_T=C(d,{ p},q,\|b\|_{\mL^{ q_1}_{ p_1}(T)},T)$ satisfying $\lim_{T\to0} C_T=0$  such that
		\begin{align}
	\Vert u\Vert_{\mathbb{H}_{\alpha, p}^\infty(T)}\leq C_T\Vert f\Vert_{\mathbb{L}^{q}_{p}(T)}.\label{esss2}
	\end{align}
Furthermore, we have
	\begin{align}
	\|u\|_{\mL^\infty(T)}\leq \hat C_T\|f\|_{\mL^q_{ p}(T)}, \quad\text{if}\quad 2/q+d/p<2,\label{es2}
	\end{align}
	and
	\begin{align}
	\|\nabla u\|_{\mL^\infty(T)}\leq \hat C_T\|f\|_{\mL^q_{ p}(T)},\quad\text{if}\quad 2/q+d/p<1,   \label{es3}
	\end{align}
where $\hat C_T>0$ is a constant satisfying $\lim_{T\to0} \hat C_T=0$.
\end{theorem}
 For  $l=1,\cdots,d$, we consider the unique solution $u_b^{(l)}\in \sH_{2,p}^q(T)$ to the equation
$$
\partial_tu^{(l)}+\frac{1}{2}\sum_{i,j=1}^da_{ij}\partial_{ij}^2u^{(l)}+b\cdot\nabla u^{(l)}+b^{(l)}=0, \quad t\in[0,T],\quad u^{(l)}(T,x)=0,\quad l=1,\cdots,d.
$$
Let $U_b:=(u_b^{(l)})_{1\leq l\leq d}$ and define \begin{align}\label{pdephi}\Phi^b(t,x):=x+U_b(t,x) \text{ for } (t,x)\in[0,T]\times\mathbb{R}^{d}.\end{align} The map $\Phi^b(t,x)$ is usually called \emph{Zvonkin's transformation map} in the literature (see, e.g. \cite{Zhang2011}, \cite{XXZZ}, \cite{ZZ}, \cite{XZ2}). It was proven to be a $\mathcal{C}^1$-diffeomorphism on $\mathbb{R}^d$ for each $t\in[0,T]$ when $T$ is small. From a similar argument as in \cite{Zhang2011}, \cite{XXZZ} and \cite{Zhang2017}, based on the regularity estimates \eqref{esss2}, \eqref{es2} and \eqref{es3}, we can obtain:
\begin{lemma}\label{hom}
Under Assumption \ref{ass}, there exists a positive constant $\epsilon=\epsilon(K,d,p,q,\lambda)$ such that if $t_0-s_0\leq \epsilon$ and $\Vert b\Vert_{\mL_p^q([s_0,t_0])}\leq \lambda$, then for each $t\in[s_0,t_0]$, $x\rightarrow\Phi^b(t,x)$ is a $\mathcal{C}^1$-diffeomorphism with
\begin{align}\label{c1hommo}
\frac{1}{2}|x-y|\leq |\Phi^b(t,x)-\Phi^b(t,y)|\leq \frac{3}{2}|x-y|.
\end{align}
Moreover, let $\delta:=\frac{1}{2}-\frac{d}{2p}-\frac{1}{q}>0$, and define $\Psi^b:=(\Phi^b)^{-1}$, i.e. the inverse of $\Phi^b$. Then. the following holds:
\begin{itemize}
\item[(1)] There exists a universal positive constant $C$ such that $\Vert\nabla\Phi^b(t,\cdot)\Vert_{\infty}+\Vert\nabla{\Psi^b}(t,\cdot)\Vert_\infty\leq C$.
\item[(2)] Let $b'\in \mL_p^q(T)$ with $\Vert b'\Vert_{\mL_p^q([s_0,t_0])}\leq \lambda$ and $\Phi^{b'}$ defined as above. Then we have for some $N>0$,
\begin{align*}
\Vert\Phi^b-&\Phi^{b'}\Vert_{\mL^\infty([s_0,t_0])}+\Vert\nabla\Phi^b-\nabla\Phi^{b'}\Vert_{\mL_p^q([s_0,t_0])}\leq N\Vert b-{b'}\Vert_{\mL_p^q([s_0,t_0])}.
\end{align*}
Here $N$ depends on $\lambda, K, p,q,d,\delta$ and the modulus of continuity of $(a_{ij})_{1\leq i,j\leq d}$.
 \item[(3)] Let $\tilde\sigma^b(t,y):=[\nabla\Phi^b\cdot\sigma]\circ({\Psi^b}(t,y))$ and similarly define $\tilde\sigma^{b'}$ with $\Vert b'\Vert_{\mL_p^q([s_0,t_0])}\leq \lambda$.  Then  $\tilde\sigma^b$ satisfies {\bf($H_{\tilde\sigma^b}$)}, $\tilde\sigma^{b'}$ satisfies {\bf($H_{\tilde\sigma^{b'}}$)}, and there exists a positive constant $N$ depending on $\lambda,K,p,q,d,\delta$, and the modulus of continuity of $(a_{ij})_{1\leq i,j\leq d}$ such that
\begin{align*}
\Vert \tilde\sigma^b-\tilde\sigma^{b'}\Vert_{\mL_p^q([s_0,t_0])}\leq N\Vert b-{b'}\Vert_{\mL_p^q([s_0,t_0])}.
\end{align*}
\end{itemize}
\end{lemma}
\begin{remark}
Using the well-posedness result in Theorem \ref{pde}, the proof of the stated lemma use now standard techniques, explained in a series of works. We give an outline of the way on how to get it. First,  based on the regularity estimates from Theorem \ref{pde} (see also \cite[Theorem 3.2]{XXZZ}, by the same argument of obtaining \cite[Theorem 6.1]{Zhang2017}, correspondingly we can also get the same results for $p,q\in[2,\infty)$ and $\frac{d}{p}+\frac{2}{q}<1$. Then, following the proof of \cite[Lemma 6.2]{Zhang2017}, we can show the above conclusion.
\end{remark}

\subsection*{Acknowledgements}
\label{sec:acknowledgements}

CL, SR and MS acknowledge financial support by the DFG via Research Unit FOR 2402.

	\end{document}